\documentclass[a4paper]{amsart}

\usepackage{xr-hyper}
\usepackage[colorlinks=true,linkcolor=blue!60!black,citecolor=teal!80!black
, linktocpage=true]{hyperref}
\usepackage{fixmath}

\usepackage[utf8]{inputenc}
\usepackage[english]{babel}

\usepackage{lmodern}

\usepackage[margin=2cm]{geometry} 
\usepackage{fancyhdr} 

\usepackage{cite} 
\usepackage{natbib} 
\usepackage{xcolor}
\usepackage{booktabs}
\usepackage{amsmath,amsfonts,amssymb,amsthm,stackrel,mathrsfs}
\usepackage{mathtools}
\usepackage{bm}
\usepackage{thmtools}
\usepackage{thm-restate}
\usepackage{spectralsequences}
\usepackage{bbm}

\usepackage{enumitem}

\usepackage{float}

\usepackage{psfrag}
\usepackage{amscd}
\usepackage{comment}
\usepackage[all]{xy}
\usepackage{url}

\usepackage{booktabs}

\usepackage{array}
\usepackage{multirow}

\usepackage{mathrsfs}

\usepackage{changepage,parskip}

\newcommand{\ZZ}{\mathbb{Z}}
\newcommand{\QQ}{\mathbb{Q}}

\newcommand{\VV}{\mathbb{V}}
\newcommand{\PP}{\mathbb{P}}
\newcommand{\CC}{\mathbb{C}}

\newcommand{\OO}{\mathbb{O}}

\newcommand{\DD}{\mathbb{D}}

\renewcommand{\AA}{\mathbb{A}}
\newcommand{\ee}{\mathbbm{e}}
\newcommand{\e}{\bm{e}}

\newcommand{\A}{\mathcal{A}}
\newcommand{\C}{\mathcal{C}}
\renewcommand{\O}{\mathcal{O}}

\newcommand{\J}{\mathcal{J}}

\newcommand{\M}{\mathcal{M}}

	\renewcommand{\P}{\mathcal{P}}
	\newcommand{\SSS}{\mathfrak{S}}
	
	\newcommand{\un}{\underline}
	\newcommand{\ov}{\overline}
	
	\newcommand{\wt}{\widetilde}

	\newcommand{\s}{\mathfrak{s}}

	\newcommand{\Mbargn}{\overline{\mathcal{M}}_{g,n}}
	\newcommand{\Ggn}{G_{g,n}}

	\newcommand{\Sym}{\operatorname{Sym}}

	\newcommand{\Res}{\operatorname{Res}}

	\newcommand{\rk}{\operatorname{rk}}
	\newcommand{\BCon}{\operatorname{BCon}}

	\newcommand{\Div}{\operatorname{Div}}
	\newcommand{\Pic}{\operatorname{Pic}}

	\newcommand{\Aut}{\operatorname{Aut}}

	\newcommand{\Simp}{\operatorname{Simp}}

	\usepackage{graphicx}
	\usepackage{epstopdf}
	\usepackage{xcolor}
	\usepackage{rotating}
	\usepackage{tikz-cd}
	\usetikzlibrary{positioning,decorations.text,quotes}
	\usetikzlibrary{arrows,babel}
	
	\usepackage{nameref}

	\usepackage{chngcntr} 
	\usepackage{setspace} 
	\usepackage{relsize} 
	\usepackage{extdash} 
	\usepackage{verbatim} 
	\usepackage{pdfsync} 
	\usepackage{stmaryrd} 
	\usepackage{marvosym} 
	\usepackage{colonequals} 
	\usepackage{float} 
	\usepackage{datetime} 
	\usepackage{url} 
	\usepackage{multicol} 
	\usepackage{color} 
	\usepackage{import} 
	\usepackage[stable]{footmisc} 
	\usepackage{csquotes} 
	
	\usetikzlibrary{cd}
	\usepackage{xcolor}
	\usetikzlibrary{shapes}
	\usetikzlibrary{decorations.markings}

	\newtheorem{theorem}{Theorem}[section]
	\newtheorem{corollary}[theorem]{Corollary}
	\newtheorem{lemma}[theorem]{Lemma}
	
	\newtheorem*{fact*}{Fact}
	\newtheorem{proposition}[theorem]{Proposition}
	\newtheorem{proposition-definition}[theorem]{Proposition-Definition}
	\newtheorem{lemma-definition}[theorem]{Lemma-Definition}
	\theoremstyle{definition}
	
	\newtheorem{remark}[theorem]{Remark}

	\numberwithin{equation}{section}

	\usepackage{xpatch}
	\let\sectionnotoc\section
	\xpatchcmd{\sectionnotoc}{{1}}{{1001}}{}{}
	\let\subsectionnotoc\subsection
	\xpatchcmd{\subsectionnotoc}{{2}}{{1001}}{}{}
	
\begin{document}
	
\title{Cohomology of the fine compactified universal Jacobians $\ov J_{2}$ and $\ov J_{2,1}$}

\author{Marco Fava}
\address{Marco Fava, Mathematics Institute, University of Warwick, Coventry, CV4 7AL, United Kingdom}
\email{marco.fava@warwick.ac.uk}

\begin{abstract}
	We compute the rational cohomology of the fine compactified universal Jacobians over the moduli spaces of stable curves of genus $2$ with $0$ and $1$ marked points.
\end{abstract}

\maketitle
\bigskip

\tableofcontents

\section{Introduction}
The universal Jacobian $J^\chi_{g,n}$ over $\ov \M_{g,n}$ parametrises pairs $(C,p_1,\ldots,p_n,L)$ given by a stable pointed curve $(C,p_1,\ldots,p_n)$ and a line bundle $L$ of characteristic $\chi$ on $C$. This, in general, is not proper. A \emph{fine compactified universal Jacobian} $\ov J^\chi_{g,n}$ of type $(g,n)$ and characteristic $\chi$ over $\Mbargn$ is a proper open substack of $\Simp^\chi_{g,n}$, the rigidified stack of relative simple torsion-free rank $1$ sheaves of characteristic $\chi$ on the universal curve $\C_{g,n}/\Mbargn$.

The study of compactified universal Jacobians started over thirty years ago with the work of Caporaso \cite{caporaso}, and was subsequently developed by many authors. Recently, a complete classification was achieved in \cite{fava2024}, building on the work of Pagani--Tommasi \cite{PT24} and Viviani \cite{viviani2023new}. For the classification in the general non-fine case and a thorough treatment of the history of this topic, we refer to \cite{FPV3} and the references therein.

The geometry of compactified universal Jacobians depends on the choice of the characteristic $\chi$ and of some stability condition $\sigma$ (see \cite[Corollary~6.19]{Kass_2019} or \cite[Theorem~D]{FPV3}). However, \cite[Theorem~1.7]{MSV} and, more recently, \cite[Theorem~1.2]{PPSW} showed that the rational cohomology groups of a fine compactified universal Jacobian $\ov J^\chi_{g,n}(\sigma)$ are independent of both the characteristic and the stability condition. For this reason, $\ov J_{g,n}$ will denote any compactified universal Jacobian of type $(g,n)$ for the remainder of the paper. (In the general non-fine case, the same holds by \cite[Theorem~1]{wu2026supporttheoremuniversalcompactified}.) 

Fine compactified universal Jacobians $\ov J_{0,n}$ are trivially isomorphic to the moduli spaces of curves $\ov \M_{0,n}$.
The rational cohomology of the fine compactified universal Jacobians $\ov J_{1,n}^\chi$ was computed by Pagani--Tommasi in \cite{PTgenus1}, via comparison with the rational cohomology of $\ov\M_{1,n+1}$.
The purpose of this paper is to carry out the first explicit computations in genus $2$. 

We determine the rational cohomology and Chow groups of the fine compactified universal Jacobians $\ov J_2$ over $\overline{\mathcal M}_2$ and $\ov J_{2,1}$ over $\overline{\mathcal M}_{2,1}$. Indeed, such moduli spaces exist and are unique up to isomorphism by \cite[Theorem~G]{FPV3}. Our main result is the following. Let $L=[\QQ(-1)]\in K_0(\mathrm{MHS}_{\QQ})$ denote the Tate class. 
\begin{theorem}\label{T:CohJ2} (Proposition~\ref{P:eJ2}, Proposition~\ref{P:eJ21}, Lemma~\ref{L:chow}) 
	\begin{enumerate}
		\item The fine compactified universal Jacobians $\ov J_2$ and $\ov J_{2,1}$ have Tate type rational cohomology given by the Hodge--Euler characteristics
		$$
		\e_{HE}(\ov J_2)=L^5+3L^4+6L^3+6L^2+3L+1,
		$$
		and
		$$
		\e_{HE}(\ov J_{2,1})=L^6+5L^5+15L^4+20L^3+15L^2+5L+1.
		$$
		\item For $n=0,1$ and $k\in\ZZ$, the cycle class map $\operatorname{CH}^k_{\QQ}(\ov J_{2,n})\to H^{2k}(\ov J_{2,n},\QQ)$ is an isomorphism.
	\end{enumerate}
\end{theorem}

Recall that the polynomial of $\ov\M_2$ is $L^3+2L^2+2L+1$ (Mumford \cite{Mumford1983}) and the one of $\ov\M_{2,1}$ is $L^4+3L^3+5L^2+3L+1$ (Faber \cite{Faber1988} and Getzler \cite{GetzlerTopologicalRecursionGenus2}).

The proof of the Theorem is based on the simultaneous use of two stratifications: the moduli space $\overline{\mathcal M}_{g,n}$ is stratified by stable dual graphs, while each fibre of a fine compactified universal Jacobian admits a toroidal stratification indexed by connected spanning subgraphs and admissible multidegrees. 
We compute the equivariant Hodge--Euler characteristic of each stratum while retaining the action of the relevant graph automorphism group. Taking invariants and using additivity then yields the Hodge--Euler characteristic of the compactified universal Jacobian. As the latter is a smooth and proper DM-stack, its cohomology is determined by its Hodge--Euler characteristic. Finally, the fact that the cycle map is an isomorphism follows as an application of \cite[Lemma~3.11]{CanningLarson} by Canning and Larson.

The methods used in this paper essentially follow the strategy outlined at the beginning of \cite[Section~3.5]{PPSW}. While providing explicit computations for the smallest unknown cases, we lay the ground for the general formulas for the Hodge numbers of fine compactified universal Jacobians to appear in \cite{FKS}.

Full computations for the compactified Jacobian spaces relative to all the strata in the boundary of $\ov\M_{2}$ and $\ov\M_{2,1}$ are available as an ancillary file.

\subsectionnotoc{Acknowledgments.} The author wishes to thank Nicola Pagani for introducing him to the problem and dedicating time to answering his questions. He is also grateful to Terry Song and Siddarth Kannan for helpful discussions and comments on a previous version which led to some improvements. The author is supported by the EPSRC Fellowship EP/X02752X/1.

\subsectionnotoc{AI disclosure.} Artificial intelligence (GPT-5.6 Sol, made available by the University of Warwick) was used by the author exclusively for reviewing and polishing purposes.

\section{Preliminaries}\label{S:Preliminaries}

In this section we fix the notation used in the computations below. We recall the two stratifications that will be used simultaneously: the boundary stratification of $\ov\M_{g,n}$ by dual graphs and the toroidal stratification of the fibres of a fine compactified universal Jacobian induced by a PT-assignment, or equivalently by a $V$-function.  Finally, we recall the equivariant Hodge--Euler characteristic, which is the additive invariant computed in the proof of Theorem~\ref{T:CohJ2}.

\subsection{Stable graphs and boundary strata}\label{SS:BoundaryStrata}

Let $\Ggn$ be a set of representatives for the isomorphism classes of stable $n$-pointed graphs of genus $g$.  If $\Gamma\in\Ggn$, we write $V(\Gamma)$ and $E(\Gamma)$ for its vertices and edges, $g(v)$ for the genus of the vertex $v$, and $n(v)$ for the number of legs and half-edges incident to $v$.  The corresponding locally closed stratum of $\ov\M_{g,n}$ is denoted by $\M_\Gamma$; explicitly,
\begin{equation*}
	\M_\Gamma=
	\left(\prod_{v\in V(\Gamma)}\M_{g(v),n(v)}\right)\big/\Aut(\Gamma).
\end{equation*}
The clutching morphism
\begin{equation*}
	\xi_\Gamma:\wt\M_\Gamma:=\prod_{v\in V(\Gamma)}\M_{g(v),n(v)}\longrightarrow \M_\Gamma\subset \ov\M_{g,n}
\end{equation*}
is the finite quotient by $\Aut(\Gamma)$.

The stratification
\begin{equation}\label{E:TopStratification}
	\ov\M_{g,n}=\bigsqcup_{\Gamma\in\Ggn}\M_\Gamma
\end{equation}
is compatible with specialisation: if a point of $\M_{\Gamma'}$ lies in the closure of $\M_\Gamma$, then $\Gamma$ is obtained from $\Gamma'$ by contracting a set of edges.  This is the stratification over which all the relative compactified Jacobians appearing below will be computed.

\subsection{Compactified Jacobians}
By \cite[Theorem~A]{fava2024} and, more generally \cite[Theorem~A]{FPV3}, fine compactified universal Jacobians are classified by \emph{universal V-stability conditions}, or, equivalently, by \emph{V-functions} (\cite[Proposition~3.11]{FPV3}). Let $X$ be a stable curve with markings such that $[X]\in\ov\M_{g,n}$, and let $\sigma$ be a V-function of type $(g,n)$ and characteristic $\chi$. Then, by \cite[Theorem~1.26]{viviani2023new} and, more generally \cite[Theorem~3.32, Proposition~4.19]{FPV2}, $\sigma$ induces a \emph{PT-assignment} $\P_{\sigma,\Gamma_X}$ of degree $d=\chi+g-1$ on the dual graph $\Gamma_X$ of $X$ (\cite[Definition~3.9]{FPV2}). Note that in the case of a general V-stability condition, the notion of a PT-assignment coincides with that of a stability assignment of \cite[Definition~4.4]{PT24}.

We recall that a general PT-assignment of a fixed degree $d$ on $\Gamma_X$ is an upper subset of 
\begin{equation*}
	\OO_\emptyset^d(\Gamma_X)=\{(G,\underline{d}):G\leq\Gamma_X\textup{ connected and spanning, }\underline{d}\in\Div^{d-|E(G)^c|}(G)\},
\end{equation*}
satisfying the properties of \cite[Definition~3.9(1)]{FPV2}.
In particular, the fine compactified Jacobian associated to $\P_{\sigma,\Gamma_X}$ is stratified in locally closed subschemes as follows:
\begin{equation}\label{E:ToroidalStratification}
	\ov J(\P_{\sigma,\Gamma_X})=\bigsqcup_{(G,\underline{d})\in\P_{\sigma,\Gamma_X}}J_{X(E(G))}^{\underline{d}},
\end{equation}
where $J_{X(E(G))}^{\underline{d}}$ denotes the multidegree $\underline{d}$ generalised Jacobian of the partial normalisation $X(E(G))$ of $X$ at the nodes corresponding to $E(\Gamma_X)\setminus E(G)$. 
From the normalisation exact sequence of $X(E(G))$, we have that 
\begin{equation*}
	J^{\un d}_{X(E(G))}\xrightarrow{\nu^*} J^{\un d}_{\wt{X}}
\end{equation*}
is a principal $(\CC^*)^{b_1(G)}=H^1(G,\CC^*)\cong \operatorname{Hom}(H_1(G,\ZZ),\CC^*)$-bundle.

We observe that the strata of $\ov J(\P_{\sigma,\Gamma_X})$ are grouped by their codimension, which equals $b_1(\Gamma_X)-b_1(G)$. In particular, the codimension-zero strata, and hence the irreducible components, correspond to the multidegrees in $\P_{\sigma,\Gamma_X}(\Gamma_X)$, and the maximal codimension strata are indexed by the spanning trees, since $\P_{\sigma,\Gamma_X}(G)$ is a singleton whenever $b_1(G)=0$ when $\sigma$ is general, by \cite[Lemma~3.33]{FPV2}. 
Finally, for a general PT-assignment, for every connected spanning subgraph $G\leq\Gamma_X$, the natural projection $\Div(G)\twoheadrightarrow\Pic(G)$ induces a bijection
$$
\P_{\sigma,\Gamma_X}(G):=\{\underline{d}:(G,\underline{d})\in\P_{\sigma,\Gamma_X}\}\leftrightarrow\Pic^{d-|E(G)^c|}(G),
$$
and, thus, the cardinality of the strata of $\ov J(\P_{\sigma,\Gamma_X})$ indexed by $G$ equals the complexity (i.e. the number of spanning trees) of $G$.

For a fixed topological stratum $\M_\Gamma$, we denote by
\begin{equation*}
	\ov J_\Gamma:=\ov J_{g,n}\times_{\ov\M_{g,n}}\M_\Gamma
\end{equation*}
the corresponding stratum of the universal compactified Jacobian, and, similarly, we denote by 
$$
\wt J_\Gamma:=\ov J_\Gamma\times_{\M_\Gamma}\wt\M_\Gamma,
$$
the pullback. 
For every connected spanning subgraph $G\leq \Gamma$, and every $\un d\in\P_{\sigma,\Gamma}(G)$, let 
$$J_{\Gamma,E(G)}^{\un d}\subseteq \wt J_\Gamma$$
denote the relative toroidal stratum whose fibre over a curve $X$ of type $\Gamma$ is $J^{\un d}_{X(E(G))}$. The multidegree superscript will be omitted when $|\P_{\sigma,\Gamma}(G)|=1$. Combining \eqref{E:TopStratification} and \eqref{E:ToroidalStratification}, for any V-function $\sigma$, we have the decomposition
\begin{equation*}
	\ov J_{g,n}(\sigma)
	=
	\bigsqcup_{\Gamma\in\Ggn}\left(
	\bigsqcup_{(G,\un d)\in\P_{\sigma,\Gamma}}
	J^{\un d}_{\Gamma,E(G)}\right)/\Aut(\Gamma).
\end{equation*}

\subsection{Equivariant Hodge--Euler characteristics}\label{SS:HE}

Let $H$ be a finite group acting on a complex algebraic variety or Deligne--Mumford stack $Y$. Then, the cohomology groups $H^i_c(Y,\QQ)$ naturally carry the structure of an $H$-representation, and we define its $H$-equivariant Hodge--Euler characteristic by
\begin{equation}\label{E:EquivariantHE}
	\e_{HE}^{H}(Y):=
	\sum_i(-1)^i\,[H^i_c(Y,\QQ)]
	\in K_0(\mathrm{MHS}_{\QQ})\otimes R(H),
\end{equation}
where $R(H)$ is the representation ring of $H$. The product in $K_0(\mathrm{MHS}_{\QQ})\otimes R(H)$ is induced by the diagonal action of $H$. We denote by $\e_{HE}(Y)$ the image of $\e_{HE}^{H}(Y)$ in $K_0(\mathrm{MHS}_{\QQ})$ and by $\e_{HE}^{H}(Y)^H$ the coefficient of the trivial representation.

The invariant \eqref{E:EquivariantHE} is additive for locally closed decompositions and multiplicative for products.  Thus, if $Y=\bigsqcup_\alpha Y_\alpha$ is a stratification by $H$-invariant strata, then
\begin{equation*}
	\e_{HE}^{H}(Y)=\sum_\alpha \e_{HE}^{H}(Y_\alpha).
\end{equation*}

We are interested in computing the equivariant Hodge--Euler characteristic of the $\Aut(\Gamma)$-invariant toroidal strata
\[
	\bigsqcup_{(\un d,G)}J^{\un d}_{\Gamma,E(G)}\xrightarrow{\pi} \wt\M_\Gamma,
\]
via their Leray spectral sequences, which, by results of \cite{Arapura_2004}, are functorial and compatible with mixed Hodge structures. Therefore, the Hodge--Euler characteristic of the total space of a stratum is computed as the alternating sum of the compactly supported cohomology groups of the base with coefficients in the local systems $R^q\pi_{!}\QQ$.


Summing first over the toroidal strata and then over the boundary strata gives
\begin{equation*}
	\e_{HE}(\ov J_{g,n})
	=\sum_{\Gamma\in \Ggn}\e_{HE}(\ov J_\Gamma)=
	\sum_{\Gamma\in\Ggn}
	(\e_{HE}^{\Aut(\Gamma)}(\wt J_\Gamma))^{\Aut(\Gamma)}.
\end{equation*}

For $n>0$ we keep track of the natural $\SSS_n$-action by writing the equivariant Hodge--Euler characteristic in
\[
	K_0(\mathrm{MHS}_{\QQ})\otimes \Lambda_n,
\]
where $\Lambda_n$ is the ring of degree $n$ symmetric functions and is identified with the representation ring of $\SSS_n$ via Schur polynomials (see \cite[Theorem~4.8]{Getzler96}).

\section{Proof of the Theorem}\label{S:Proof}
Since the cohomology of a fine compactified universal Jacobian $\ov J(\sigma)$ is independent of the V-function, in this section we fix the general V-function of characteristic $\chi=-1$
$$
\sigma\equiv 0\colon\DD_{g,n}\to\ZZ,
$$
well defined for each pair $(g,n)$, and we set $\ov J_{g,n}:=\ov J(\sigma)$.

We prove Theorem~\ref{T:CohJ2} by computing the Hodge--Euler characteristics of $\ov J_2$ and $\ov J_{2,1}$ in Proposition~\ref{P:eJ2} and Proposition~\ref{P:eJ21}, and addressing the isomorphism with the Chow groups in Lemma~\ref{L:chow}. The following Lemma will be useful.
\begin{lemma}
	\label{L:prod}
	Let $\pi_i\colon \mathcal C_i\to S$ be families of connected nodal curves,
	with sections $p_i\colon S\to\mathcal C_i$ through the relative smooth loci, and
	let
	\[
	\mathcal C=\mathcal C_1\cup_{p_1\sim p_2}\mathcal C_2
	\]
	be the family obtained by gluing the two sections. Let $\sigma$ be a general
	stability condition on $\mathcal C/S$, and let $\sigma_i$ be the induced
	stability condition on $\mathcal C_i/S$.
	Then restriction induces an isomorphism of relative fine compactified
	Jacobian spaces over $S$:
	\[
	\rho\colon \ov J_{\C/S}(\sigma)
	\xrightarrow{\sim}
	\ov J_{\C_1/S}(\sigma_1)
	\times_S
	\ov J_{\C_2/S}(\sigma_2).
	\]
\end{lemma}

\begin{proof}
	Since $\sigma$ is general, any $I\in\ov J_{\C/S}(\sigma)$ is simple, and hence locally free along the node section. Therefore, $I_{|_{\C_i}}$ is a torsion free rank one sheaf on $\C_i/S$. Moreover, by \cite[Lemma~5.3(ii)]{FPV1}, $I_{|_{\C_i}}$ is $\sigma_i$-stable.
	
	Vice-versa, let $I_i\in\ov J_{\C_i/S}(\sigma_i)$. Up to twisting by a line bundle on $S$, we may identify $I_i$ with $I_i\otimes\pi_i^*(p_i^*I_i)^{-1}$ and assume that $I_1$ and $I_2$ glue uniquely (in the rigidification) to a sheaf $I$ on $\C$. Finally,  up to passing to the complementary subcurve, any $Y\in\BCon(C)$ is such that $Y=\C_i$ or $Y\in\BCon(\C_i)$. In either case, since $\sigma_i$ is induced by $\sigma$, by \cite[Definition~5.1]{FPV1}, we conclude that $I\in\ov J_{\C/S}(\sigma)$.
\end{proof}

\begin{remark}
	\label{R:prod-equivariant}
	Let $G= G_1\cup \{e\}\cup G_2\leq\Gamma_1\cup\{e\}\cup\Gamma_2=\Gamma$, where $G_i:=\Gamma_i\cap G$, be a connected spanning subgraph.
	Applying the lemma to a stratum $S=\wt\M_\Gamma$, and taking the equivariant Hodge--Euler characteristic, we obtain
	$$
	\e_{HE}^{\operatorname{Stab}(G,e)}\left(J_{\Gamma,E(G)}\right)
	= \e_{HE}^{\operatorname{Stab}(G,e)}\left(J_{\Gamma_1,E(G_1)}
	\times_{\wt\M_{\Gamma}}J_{\Gamma_2,E(G_2)}\right),
	$$
	where $\operatorname{Stab}(G,e)$ denotes the subgroup of the stabiliser of $G$ that preserves the edge $e$.
	
	In particular, if $e$ is fixed by the automorphism group of $\Gamma$, we have 
	$$
	\e_{HE}(\ov J_\Gamma)=\e_{HE}^{\Aut(\Gamma)}\left(\wt J_{\Gamma_1}\times_{\wt\M_\Gamma}\wt J_{\Gamma_2}\right)^{\Aut(\Gamma)}.
	$$
\end{remark}

Now, we are ready to prove the main propositions. The equivariant Hodge--Euler characteristics of the $\M_{g,n}$ involved in the computations are computed in Jonas Bergstr\"om's repository \url{https://github.com/jonasbergstroem/Cohomology-of-moduli-spaces-of-curves}.
\begin{proposition}\label{P:eJ2}
	The Hodge--Euler characteristic of $\ov J_{2}$ is 
	$$
	\e_{HE}(\ov J_2)=L^5+3L^4+6L^3+6L^2+3L+1.
	$$
\end{proposition}

\begin{proof}
	Since $\ov J_2$ is orbifold-smooth and proper, its cohomology is pure and satisfies Poincar\'e duality. Thus, it is enough to compute the components of $\e_{HE}(\ov J_2)$ in weights at least $\dim_{\CC}(\ov J_2)$.
	We do so by studying the equivariant Hodge--Euler characteristic of 
	$$
	\ov J_\Gamma\longrightarrow\M_\Gamma
	$$
	for the following strata.
	\begin{itemize}
		\item \textbf{Stratum $\M_2$.} Let $C$ be a generic curve of genus $2$, and let 
		$$
		AJ\colon\Sym^2(C)\rightarrow \Pic^2(C)
		$$
		be the Abel--Jacobi map, whose fibres are of the form $\PP H^0(C,\O_C(p+q))$, for $p,q\in C$. By Riemann--Roch, we have that $h^0(C,\O_C(p+q))=2$ if $\O_C(p+q)$ is the canonical sheaf, and it equals $1$ otherwise. Now, if $\C_2\cong\M_{2,1}$ is the universal curve, and $J^2_2$ the universal Jacobian in degree $2$, the following diagram commutes:
		\begin{equation}\label{E:AJmapM2}
			\begin{tikzcd}
				\C_2^2/\SSS_2 \arrow{r}{AJ} \arrow[dr] & J^2_{2} \arrow [d]\\ 
				& \M_{2}
			\end{tikzcd}
		\end{equation}
		By the preceding description of the Abel--Jacobi map, we have the decomposition $\C_2^2=U\sqcup Z$, where $Z$ is the preimage of the canonical divisor with respect to the Abel--Jacobi map, and $U$ is its complement. Since $AJ(Z)$ is a $\M_2$-point in $J^2_2$, we deduce that
		$$
		\e_{HE}(J_2)=\e_{HE}(J^2_2)=\e_{HE}(\C_2^2)-\e_{HE}(Z)+\e_{HE}(\M_2),
		$$
		where $J_2=J^0_2$.
		Since $\Sym^2\C_2=\M_{2,2}/\SSS_2\sqcup\M_{2,1}$ and since $Z\to\M_2$ is a $\PP^1$-bundle, we obtain
		\begin{align*}
			\e_{HE}(J_2)&=\e_{HE}^{\SSS_2}(\M_{2,2})^{\SSS_2}+\e_{HE}(\M_{2,1})-L\e_{HE}(\M_2)=\\&=((L^5+L^4-1)\s_2+L^4\s_{1^2})^{\SSS_2}+(L^4+L^3)-L^4=L^5+L^4+L^3-1.
		\end{align*}

		\item\textbf{Stratum $\M_{1,2}/\SSS_2$.} In genus $1$, the Abel--Jacobi map 
		\begin{equation}\label{E:AJmapM13}
		\begin{tikzcd}
			\C_{1,2} \arrow[r,"AJ"] \arrow[dr]&J^1_{1,2} \arrow[d] \\
			&\M_{1,2}
		\end{tikzcd}
		\end{equation}
		is an isomorphism. After translating, we can identify $J^1_{1,2}$ with $J^{-1}_{1,2}$, which is contained in the compactified Jacobian $\ov J_2$. The universal curve stratifies as 
		$$
		\C_{1,2}=\M_{1,3}\sqcup\M_{1,2}\sqcup\M_{1,2},
		$$
		where the two copies of $\M_{1,2}$ are obtained by the sections corresponding to the marked points.
		Thus, we have, as $\SSS_2$-representation,
		\begin{align*}
			\e_{HE}^{\SSS_2}(\C_{1,2})&=\Res^{\SSS_3}_{\SSS_2}\e_{HE}^{\SSS_3}(\M_{1,3})+\e_{HE}(\M_{1,2})(\s_2+\s_{1^2})=\\
			&=(L^3\s_2-\s_{1^2})+L^2(\s_2+\s_{1^2}).
		\end{align*}
		First, we consider the codimension $0$ stratum
		\begin{equation}\label{E:J12}
			1\longrightarrow\CC^*\longrightarrow J_{\Gamma,E(\Gamma)}\longrightarrow J_{1,2}\longrightarrow1
		\end{equation}
		of the compactified Jacobian. Its $\CC^*$ component has cohomology with compact support
		$$
		H^i_c(\CC^*,\QQ)=\begin{cases}
			\QQ\s_{1^2},\textup{ if }i=1,\\
			\QQ(-1)\s_2.\textup{ if }i=2,\\
			0,\textup{ otherwise},
		\end{cases}
		$$
		and is invariant with respect to the action by translation induced by $\pi_1(\M_{1,2})$ on $\CC^*$, since a connected group acts trivially on its cohomology. Thus, Leray-Hirsch applies to the $\SSS_2$-equivariant Leray spectral sequence associated to \eqref{E:J12}, and we can compute the Hodge--Euler characteristic of the generalised Jacobian as
		\begin{align*}
			\e_{HE}^{\SSS_2}(J_{\Gamma,E(\Gamma)})^{\SSS_2}&=(\e_{HE}(\C_{1,2})\e_{HE}(\CC^*))^{\SSS_2}=\\
			&=(((L^3+L^2)\s_2+(L^2-1)\s_{1^2})(L\s_2-\s_{1^2}))^{\SSS_2}=L^4+L^3-L^2+1.
		\end{align*}
		The same argument is used in \cite[Theorem~1.2]{PPSW}.
		
		Similarly, the codimension $1$ stratum has Hodge--Euler characteristic
		$$
		\e_{HE}^{\SSS_2}(J_{\Gamma,\emptyset})^{\SSS_2}=\e_{HE}^{\SSS_2}(\C_{1,2})^{\SSS_2}=L^3+L^2.
		$$
		Hence, adding both contributions we obtain
		$$
		\e_{HE}(\ov J_{\Gamma})=L^4+2L^3+1.
		$$
		\item \textbf{Stratum $(\M_{1,1}\times\M_{1,1})/\SSS_2$.} Let $E$ denote an elliptic curve, then its rational cohomology (with compact support) is given by 
		\begin{equation*}
			H^i_c(E,\QQ)=\begin{cases*}
				\QQ,\quad \textup{for }i=0,\\
				\VV_1,\quad \textup{for }i=1,\\
				\QQ(-1),\quad \textup{for }i=2,\\
				0,\quad \textup{otherwise},
			\end{cases*}
		\end{equation*}
		where $\VV_1$ denotes the standard representation of the symplectic group $\textup{Sp}_2(\ZZ)$, and the numbers in parentheses denote the Tate twists. 
		
		The fibre over a point $[X]=[E\cup E']$ of this stratum is the product $ J_E\times J_{E'}$ of the Jacobians of two elliptic curves joined in a node. In particular, $\ov J_X= J_X$, as $X$ is of compact type. Since the Abel--Jacobi map $E\to  J_E$ is an isomorphism for any elliptic curve, the cohomology of the fibre is given by
		\begin{equation*}
			H^i_c( J_E\times  J_{E'},\QQ)=\begin{cases*}
				\QQ,\quad \textup{for }i=0,\\
				\VV_1\oplus\VV_1,\quad \textup{for }i=1,\\
				\QQ(-1)\otimes\QQ\oplus\QQ\otimes\QQ(-1)\oplus\VV_1\otimes\VV_1,\quad \textup{for }i=2,\\
				\VV_1(-1)\oplus\VV_1(-1),\quad \textup{for }i=3,\\
				\QQ(-2),\quad \textup{for }i=4,\\
				0,\quad \textup{otherwise},
			\end{cases*}
		\end{equation*}
		We observe that $H_c^\bullet(\M_{1,1},\VV_1)=0$ (see, for example \cite{PetersenGenus1}), and $H_c^\bullet(\M_{1,1},\QQ)$ equals $\QQ(-1)$ in degree $2$ and vanishes otherwise. Therefore, the entries $H^p_c(\M_{1,1}\times\M_{1,1},H^q_c( J_X,\QQ))^{\SSS_2}$ of the Leray spectral sequence associated to the fibration $ J_X\to(\M_{1,1}\times\M_{1,1})/\SSS_2$ are non-zero only for $p=4$. In particular,
		\begin{equation*}
			H^4_c(\M_{1,1}\times\M_{1,1},H^q_c( J_X,\QQ))^{\SSS_2}=\begin{cases*}
				\QQ(-2),\quad \textup{for }q=0,\\
				\QQ(-3),\quad \textup{for }q=2,\\
				\QQ(-4),\quad \textup{for }q=4,\\
				0,\quad\textup{otherwise},
			\end{cases*}
		\end{equation*}
		and we have
		\begin{equation*}
			\e_{HE}(\ov J_\Gamma)=L^4+L^3+L^2.
		\end{equation*} 
		\item \textbf{Codimension $2$ strata.} The two strata of codimension $2$ are indexed by the following graphs:
		\begin{center}\scalebox{.8}{
			\begin{tikzpicture}[remember picture, inner/.style={draw,circle,thick,node distance={15mm}},
				outer/.style={node distance={40mm}}]
				\node[outer] (A1) {
					\begin{tikzpicture}
						\node[inner] (1) {$1$};
						\node[inner, right of=1] (2) {$0$};
						\draw (1) to  (2);
						\draw (2) to [out=45, in=315, looseness=5] (2);
					\end{tikzpicture}
				};
				\node[outer, right of =A1] (A2) {
					\begin{tikzpicture}
						\node[inner] (1) {$0$};
						\draw (1) to [out=45, in=315, looseness=5] (1);
						\draw (1) to [out=135, in=225, looseness=5] (1);
					\end{tikzpicture}
				};
			\end{tikzpicture}}.
		\end{center}
		We observe that neither stratum can contribute a weight $5$ term. Indeed, for the graph on the left we deduce this from Lemma~\ref{L:prod}, while for the graph on the right we can argue as in Stratum $\M_{0,5}/\SSS_2\wr\SSS_2$.
		Therefore, each of these strata $\M_\Gamma$ contributes for a term $L^3$ for each irreducible component of $\ov J_\Gamma$ up to the action of $\Aut(\Gamma)$. Since the number of irreducible components of $\ov J_\Gamma$ equals the complexity of $\Gamma$, we have a contribution of $2L^3$.
	\end{itemize}
	Adding up all the invariant Hodge--Euler characteristics we obtain 
	$$\e_{HE}(\ov J_2)=L^5+3L^4+6L^3+\textup{ lower degree terms,}$$
	hence the claim follows by Poincar\'e duality. This is independently verified with the complete computations over the remaining strata in the ancillary file.
\end{proof}

\begin{proposition}\label{P:eJ21}
	The Hodge--Euler characteristic of $\ov J_{2,1}$ is 
	$$
	\e_{HE}(\ov J_{2,1})=L^6+5L^5+15L^4+20L^3+15L^2+5L+1.
	$$
\end{proposition}

\begin{proof}
As in Proposition~\ref{P:eJ2}, it is enough to study the equivariant Hodge--Euler characteristic of a subset of the strata
	$$
	\ov J_\Gamma\longrightarrow\M_\Gamma.
	$$
\begin{itemize}
\item \textbf{Stratum $\M_{2,1}$.} 
Consider the commutative diagram
\begin{equation}\label{E:AJmapM21}
\begin{tikzcd}
	\C^3/\SSS_2 \arrow[r] \arrow[dr] & J^2_{2,1} \arrow [d]\\ 
	 & \M_{2,1}
\end{tikzcd}
\end{equation}
obtained by pulling back Diagram~\eqref{E:AJmapM2} along the forgetful morphism $\pi: \M_{2,1}\to\M_2$. Here the power of the universal curve is taken as fibred product over $\M_{2}$. As in the case of the Stratum $\M_2$, we have the decomposition 
$$
\C^3/\SSS_2=U\sqcup Z
$$
where $Z$ denotes the preimage of the canonical divisor via the Abel--Jacobi map, and $U$ its complement. Since $Z\to \M_{2,1}$ is a $\PP^1$ bundle, its Hodge--Euler characteristic is given by 
$$
\e_{HE}(Z)=(L+1)\e_{HE}(\M_{2,1}),
$$
hence, we have
$$
\e_{HE}(J_{2,1})=\e_{HE}(J^2_{2,1})=\e_{HE}(\C^3/\SSS_2)-L\e_{HE}(\M_{2,1}).
$$
Furthermore, we observe that the universal curve is stratified by 
$$
\C^3=\M_{2,3}\sqcup (\bigsqcup_{\{i,j\}\subset[3]}\M_{2,\{i,j\}})\sqcup \M_{2,1}.
$$
Thus, observing that the $\SSS_2$-action is trivial on one copy of $\M_{2,2}$, while it swaps the other two, we obtain
\begin{align*}
	\e_{HE}(\C^3/\SSS_2)&=\e_{HE}^{\SSS_2}(\C^3)^{\SSS_2}=(\Res^{\SSS_3}_{\SSS_2}\e_{HE}^{\SSS_3}(\M_{2,3})+\e_{HE}(\M_{2,2})(\s_2+\s_{1^2}))^{\SSS_2}+\e_{HE}(\M_{2,2})+\e_{HE}(\M_{2,1})=\\
	&=((L^6+L^5-L^4-L)\s_2+(L^5-L)(\s_{2}+\s_{1^2})+((L^5+L^4-1)\s_2+L^4\s_{1^2})(\s_2+\s_{1^2}))^{\SSS_2}+\\&+(L^5+2L^4-1)+(L^4+L^3)=\\
	&=L^6+4L^5+4L^4+L^3-2L-2,
\end{align*}
and, consequently,
$$
\e_{HE}(J_{2,1}\to\M_{2,1})=L^6+3L^5+3L^4+L^3-2L-2.
$$

\item \textbf{Stratum $\M_{1,3}/\SSS_2$.} In genus $1$, the Abel--Jacobi map 
\[
\begin{tikzcd}
	\C_{1,3} \arrow[r,"AJ"] \arrow[dr]&J^1_{1,3} \arrow[d] \\
	&\M_{1,3}
\end{tikzcd}
\]
is an isomorphism. After translating, we can identify $J^1_{1,3}$ with $J^{-1}_{1,3}$, which is contained in the compactified Jacobian $\ov J_{2,1}$. The universal curve stratifies as 
$$
\C_{1,3}=\M_{1,4}\sqcup\M_{1,3}\sqcup\M_{1,3}\sqcup\M_{1,3},
$$
where the three copies of $\M_{1,3}$ are obtained by the sections corresponding to the marked points.
Thus, we have as an $\SSS_3$-representation
\begin{align*}
	\e_{HE}^{\SSS_3}(\C_{1,3})&=\Res^{\SSS_4}_{\SSS_3}\e_{HE}^{\SSS_4}(\M_{1,4})+\e_{HE}^{\SSS_3}(\M_{1,3})(\s_3+\s_{21})=\\
	&=((L^4-L^2)\s_3-L(\s_{21}+\s_{1^3})+\s_{21}+\s_3)+(L^3\s_3-s_{1^3})(\s_3+\s_{21}),
\end{align*}
with restriction as an $\SSS_2$-representation
$$
\e_{HE}^{\SSS_2}(\C_{1,3})=(L^4+2L^3-L^2-L+1)\s_2+(L^3-2L-1)\s_{1^2}.
$$

Arguing as in Stratum~$\M_{1,2}/\SSS_2$, we see that the Hodge--Euler characteristic of the generalised Jacobian in the codimension $0$ stratum is obtained by taking the $\SSS_2$-invariants after multiplying $\e^{\SSS_2}_{HE}(\C_{1,3})$ with $\e^{\SSS_2}_{HE}(\CC^*)$. Namely
\begin{align*}
	\e_{HE}^{\SSS_2}(J_{\Gamma,E(\Gamma)})^{\SSS_2}&=(\e_{HE}^{\SSS_2}(\C_{1,3})\e_{HE}^{\SSS_2}(\CC^*))^{\SSS_2}=\\&=(((L^4+2L^3-L^2-L+1)\s_2+(L^3-2L-1)\s_{1^2})(L\s_2-\s_{1^2}))^{\SSS_2}=\\&=L^5+2L^4-2L^3-L^2+3L+1.
\end{align*}

Similarly, the codimension $1$ stratum has Hodge--Euler characteristic
$$
\e_{HE}^{\SSS_2}(J_{\Gamma,\emptyset})^{\SSS_2}=\e_{HE}^{\SSS_2}(\C_{1,3})^{\SSS_2}=L^4+2L^3-L^2-L+1.
$$
Hence, adding both contributions we obtain
$$
\e_{HE}(\ov J_{\Gamma})=L^5+3L^4-2L^2+2L+2.
$$

\item \textbf{Stratum $\M_{1,2}\times\M_{1,1}$.} The generic curve on this stratum is of compact type, hence, by Lemma~\ref{L:prod}, the compactified Jacobian is the product of the compactified Jacobians $J_{1,2}\to\M_{1,2}$ and $J_{1,1}\to\M_{1,1}$. Thus, using previous computations, we have
$$
\e_{HE}(\overline J_{\Gamma})=(L^3+2L^2-1)(L^2+L)=L^5+3L^4+2L^3-L^2-L.
$$

\item \textbf{Stratum $(\M_{1,1}\times\M_{0,3}\times\M_{1,1})/\SSS_2$.} By Lemma~\ref{L:prod}, we have that
$$
\e_{HE}(\overline J_{\Gamma})=\e_{HE}^{\SSS_2}(\wt J_{1,1}\times J_{0,3}\times\wt J_{1,1})^{\SSS_2}=\e_{HE}^{\SSS_2}(\ov J_{1,1}\times\ov J_{1,1})^{\SSS_2},
$$
since $J_{0,3}\to\M_{0,3}$ is trivial. As $\Aut(\Gamma)=\SSS_2$, the Jacobian over this stratum contributes as the one over $(\M_{1,1}\times\M_{1,1})/\SSS_2$, thus

$$
\e_{HE}(\ov J_{\Gamma})=L^4+L^3+L^2.
$$

\item \textbf{Stratum $\M_{1,1}\times\M_{0,4}/\SSS_2$.} The generic curve in this stratum is the union of a nodal irreducible rational curve $X$ and an elliptic curve $E$ meeting at a point. The compactified Jacobian $\ov J_X$ is stratified by $J_{X,\{e\}}\cong \CC^*$ and $J_{X,\emptyset}\cong 1$, where $e$ is the loop of $\Gamma$. 

The cohomology with compact support of $\M_{0,4}$ as a representation of $\SSS_2$ is 
\begin{equation*}
	H^i_c(\M_{0,4},\QQ)=\begin{cases*}
		\QQ\s_{2}\oplus\QQ\s_{1^2},\textup{ if }i=1,\\
		\QQ(-1)\s_2,\textup{ if }i=2,\\
		0, \textup{ otherwise}.
	\end{cases*}
\end{equation*}
In particular, any Leray spectral sequence over this base must degenerate at the second page for degree reasons.  Therefore, we have
\begin{align}\label{E:M04S2}
	\e_{HE}^{\SSS_2}(\wt J_{\Gamma[X]})&=(\e_{HE}^{\SSS_2}(J_{X,\{e\}})+\e_{HE}^{\SSS_2}(J_{X,\emptyset}))\e_{HE}^{\SSS_2}(\M_{0,4})=\\
	&=(L\s_2-\s_{1^2}+\s_2)(L\s_2-\s_2-\s_{1^2})=L^2\s_2-2L\s_{1^2}.\notag
\end{align}
Taking the invariants with respect to $\Aut(\Gamma[X])=\Aut(\Gamma)=\SSS_2$ and multiplying by $\e_{HE}(\ov J_{\Gamma[E]})=L^2+L$, we obtain
$$
\e_{HE}(\ov J_{\Gamma})=L^4+L^3.
$$

\item \textbf{Stratum $\M_{1,2}\times(\M_{0,3}/\SSS_2)$.} The generic curve of this stratum is the union of an elliptic curve $E$ and an irreducible rational nodal curve $X$ meeting at a point. By Lemma~\ref{L:prod} and previous computations, we have 
\begin{align*}
	\e_{HE}(\ov J_{\Gamma})&=(\e_{HE}^{\SSS_2}(\wt J _{\Gamma[E]})\e_{HE}^{\SSS_2}(\wt J _{\Gamma[X]}))^{\SSS_2}=\\
	&=(L^3+2L^2-1)(L+1)=L^4+3L^3+2L^2-L-1.
\end{align*}

\item \textbf{Stratum $\M_{0,5}/\SSS_2\wr\SSS_2$.} Let $X$ be the generic curve of this stratum, then the poset of graphs in $\P_{\sigma,\Gamma}$ is the following:

\begin{center}
	\scalebox{.7}{
		\begin{tikzpicture}[
			inner/.style={
				draw,
				circle,
				thick,
				minimum size=6mm,
				inner sep=0.5pt,
				font=\small
			},
			outer/.style={
				minimum width=15mm,
				minimum height=15mm,
				inner sep=0pt
			},
			degeneration/.style={
				->,
				thick,
				shorten <=0.3mm,
				shorten >=0.3mm,
				>=stealth
			}
			]
			\node[outer] (A1) at (0,0) {
				\begin{tikzpicture}
					\node[inner] (v) {$0$};
					\draw (v) to[out=45,in=315,looseness=5] (v);
					\draw (v) to[out=134,in=225,looseness=5] (v);
				\end{tikzpicture}%
			};
			
			\node[outer] (B1) at (2.7,-2.25) {
				\begin{tikzpicture}
					\node[inner] (v) {$0$};
					\draw (v) to[out=45,in=315,looseness=5] (v);
				\end{tikzpicture}%
			};
			
			\node[outer] (B2) at (-2.7,-2.25) {
				\begin{tikzpicture}
					\node[inner] (v) {$0$};
					\draw (v) to[out=134,in=225,looseness=5] (v);
				\end{tikzpicture}%
			};
			
			\node[outer] (C1) at (0,-4.5) {
				\begin{tikzpicture}
					\node[inner] {$0$};
				\end{tikzpicture}%
			};

			\draw[degeneration] (B1) -- (A1);
			\draw[degeneration] (B2) -- (A1);
			\draw[degeneration] (C1) -- (B1);
			\draw[degeneration] (C1) -- (B2);
		\end{tikzpicture}
	}
\end{center}
For every $G$ in the poset, we have $\P_{\sigma,\Gamma}(G)=\{b_1(G)-2\}$.

Since $X$ is irreducible, all the strata of the toroidal stratification of $\ov J_X$ are connected. We consider, first, the stratum $J^0_X$ of codimension $0$. The normalisation exact sequence defining the generalised Jacobian is 
\begin{equation*}
	1\to\CC^*\times\CC^*\to J^0_X\to1,
\end{equation*}
since the normalisation $\wt{X}$ of $X$ is a rational curve, and, thus, has trivial Jacobian. Let us denote the three copies of $\SSS_2$ appearing in the wreath product $\SSS_2\wr\SSS_2$ by
\begin{equation*}
	\SSS_2\wr\SSS_2:=\SSS_2\ltimes(\SSS_2^1\times\SSS_2^2),
\end{equation*}
with non-trivial elements $\sigma,\sigma^1,\sigma^2$ respectively.
Let $[\ee_1],[\ee_2]$ be the classes of the loops of $\Gamma$ respectively generating the two copies of $\CC^*$ in the sequence. Then, we have the following actions:
\begin{align*}
	\sigma\cdot([\ee_1],[\ee_2])&=([\ee_2],[\ee_1]);\\ 	\sigma^1\cdot([\ee_1],[\ee_2])&=([-\ee_1],[\ee_2]);\\	\sigma^2\cdot([\ee_1],[\ee_2])&=([\ee_1],[-\ee_2]).
\end{align*}
We identify $\SSS_2\wr\SSS_2$ with the dihedral group 
$$
D_8:=\langle a,x\colon a^4=x^2=1,xax=a^-1\rangle,
$$
by setting $a=\sigma\sigma^1$ and $x=\sigma^1$.
Therefore, by analysing the characters of the induced action in cohomology, we deduce that the rational cohomology groups of $ J^0_{X}$ as representations of $\SSS_2\wr\SSS_2$ are:
\begin{equation}\label{Js2wrs2}
	H^i_c( J_X^{0},\QQ)=\begin{cases*}
		\QQ\langle a\rangle\textbf{ker}, \textup{ if }i=2,\\
		\QQ(-1)\textbf{2dim},\textup{ if }i=3,\\
		\QQ(-2)\textbf{triv},\textup{ if }i=4,\\
		0, \textup{ otherwise},
	\end{cases*}
\end{equation}
with notation as in \url{https://groupprops.subwiki.org/wiki/Dihedral_group:D8#Character_table}, where $\langle\ldots\rangle$\textbf{ker} denotes the one dimensional representation with a given kernel.

The group $\SSS_2\wr\SSS_2\leq \SSS_5$ acts on $\M_{0,5}$ by permuting the $4$ markings corresponding to the half edges. The cohomology with compact support of $\M_{0,5}$ as $\SSS_5$-representation is 
$$
H^i_c(\M_{0,5},\QQ)=\begin{cases*}
	\QQ\s_{311},\textup{ if }i=2,\\
	\QQ(-1)\s_{32},\textup{ if }i=3,\\
	\QQ(-2)\s_5,\textup{ if } i=4,\\
	0,\textup{ otherwise},
\end{cases*}
$$
which restricts as $\SSS_2\wr\SSS_2$-representation to
$$
	H^i_c(\M_{0,5},\QQ)\downarrow_{\SSS_2\wr\SSS_2}=\begin{cases}
		\QQ(\textbf{2dim}^{\oplus 2}\oplus \langle a^2,x\rangle \textbf{ker}\oplus \langle a\rangle\textbf{ker}),\textup{ if }i=2,\\
		\QQ(-1)(\textbf{triv}\oplus \textbf{2dim}\oplus\langle a^2,x\rangle \textbf{ker}\oplus \langle a^2,ax\rangle\textbf{ker}),\textup{ if }i=3,\\
		\QQ(-2)\textbf{triv},\textup{ if } i=4,\\
		0,\textup{ otherwise},
	\end{cases}
$$

Since the cohomology of the fibre is given by \eqref{Js2wrs2}, we observe that all the differentials in the second page of the Leray spectral sequence associated to the fibration 
$$
J_{\Gamma,E(\Gamma)}\to \M_{0,5}
$$
must vanish. Indeed the Tate twists of domain and image of any such differential do not coincide. Hence, the K\"unneth formula gives 
\begin{equation*}
		H^i_c( J_{\Gamma,E(\Gamma)},\QQ)=\begin{cases*}
		\QQ(\textbf{triv}\oplus\textbf{2dim}^{\oplus 2}\oplus\langle a^2,ax\rangle\textbf{ker}), \textup{ if }i=4,\\
		\QQ(-1)(\textbf{triv}^{\oplus 2}\oplus\textbf{2dim}^{\oplus 3}\oplus\langle a\rangle\textbf{ker}^{\oplus 3}\oplus\langle a^2,ax\rangle\textbf{ker}^{\oplus 3}\oplus\langle a^2,x\rangle\textbf{ker}^{\oplus 3}),\textup{ if }i=5,\\
		\QQ(-2)(\textbf{triv}\oplus\textbf{2dim}^{\oplus 5}\oplus\langle a\rangle\textbf{ker}^{\oplus 3}\oplus\langle a^2,ax\rangle\textbf{ker}\oplus\langle a^2,x\rangle\textbf{ker}^{\oplus 2}),\textup{ if }i=6,\\
		\QQ(-3)(\textbf{triv}\oplus\textbf{2dim}^{\oplus 2}\oplus\langle a^2,ax\rangle\textbf{ker}\oplus\langle a^2,x\rangle\textbf{ker}),\textup{ if }i=7,\\
		\QQ(-4)\textbf{triv},\textup{ if }i=8,\\
		0,\textup{ otherwise}.
	\end{cases*}
\end{equation*} 
Hence, by taking the $\SSS_2\wr\SSS_2$-invariant part we obtain
\begin{equation*}
	\e_{HE}^{\SSS_2\wr\SSS_2}(J_{\Gamma,E(\Gamma)})^{\SSS_2\wr\SSS_2}=L^4-L^3+L^2-2L+1.
\end{equation*}

The strata $J_{X(\{e_1\})}$ and $J_{X(\{e_2\})}$ of codimension $1$ are swapped by the action of $\SSS_2$. Hence, the contribution of this stratum is given by the $\SSS_2\times \SSS_2$-invariant part of the fibration over $\M_{0,5}$ with fibre $J_{X(\{e_1\})}$. We have that

$$
H^i_c(\M_{0,5},\QQ)\downarrow_{\SSS_2\times\SSS_2}=\begin{cases}
	\QQ(\s_2\times\s_2\oplus(\s_{1^2}\times\s_2)^{\oplus2}\oplus(\s_2\times\s_{1^2})^{\oplus2}\oplus\s_{1^2}\times\s_{1^2}),\textup{ if }i=2,\\
	\QQ(-1)((\s_2\times\s_2)^{\oplus 2}\oplus\s_{1^2}\times\s_2\oplus\s_2\times\s_{1^2}\oplus\s_{1^2}\times\s_{1^2}),\textup{ if }i=3,\\
	\QQ(-2)\s_2\times\s_2,\textup{ if }i=4,\\
	0,\textup{ otherwise},
\end{cases}
$$
and
$$
H^i_c(J_{X(\{e_1\})},\QQ)=\begin{cases}
	\QQ\s_{1^2}\times\s_2,\textup{ if }i=1,\\
	\QQ(-1)\s_2\times\s_2,\textup{ if }i=2,\\
	0,\textup{ otherwise}.
\end{cases}
$$
Again, the spectral sequence degenerates at the second page. Hence, taking the invariant part after applying the K\"unneth formula, we obtain
$$
\e_{HE}^{\SSS_2\wr\SSS_2}(J_{\Gamma,\{e_2\}}\sqcup J_{\Gamma,\{e_1\}})^{\SSS_2\wr\SSS_2}=L^3-2L^2+2L-2.
$$ 
Finally, since the fibre over the stratum $J_{\Gamma,\emptyset}$ is a point, we have 
$$
J_{\Gamma,\emptyset}\cong\M_{0,5},
$$
and the resulting Hodge--Euler characteristic contribution is
$$
\e_{HE}^{\SSS_2\wr\SSS_2}(J_{\Gamma,\emptyset})^{\SSS_2\wr\SSS_2}=L^2-L
$$
Therefore, summing over all the strata, we obtain
$$
\e_{HE}(\ov J_\Gamma)=L^4-L-1.
$$

\item \textbf{Stratum $(\M_{1,2}\times\M_{0,3})/\SSS_2$.} Let $X$ be the generic curve of this stratum, then the poset associated to the compactified Jacobian $\ov J_X(\sigma)$ is the following:

\begin{center}\scalebox{.8}{
		\begin{tikzpicture}[remember picture, inner/.style={draw,circle,thick,node distance={15mm}},
			outer/.style={node distance={40mm}}]
			\node[outer] (A1) {
				\begin{tikzpicture}
					\node[inner, label=135:$1$] (1) {$1$};
					\node[inner, right of=1, label=45:$-1$] (2) {$0$};
					\node[inner, draw=none, above right=0.2 of 2] (2l) {};
					\draw (1) to [out=45, in=135, looseness=1] (2);
					\draw (1) to [out=315, in=225, looseness=1] (2);
					\draw (2) to (2l);
				\end{tikzpicture}
			};
			\node[outer, right of =A1] (A2) {
				\begin{tikzpicture}
					\node[inner, label=135:$0$] (1) {$1$};
					\node[inner, right of=1, label=45:$0$] (2) {$0$};
					\node[inner, draw=none, above right=0.2 of 2] (2l) {};
					\draw (1) to [out=45, in=135, looseness=1] (2);
					\draw (1) to [out=315, in=225, looseness=1] (2);
					\draw (2) to (2l);
				\end{tikzpicture}
			};
			\node[outer, below=.9 of A1] (B1) {
				\begin{tikzpicture}
					\node[inner, label=135:$0$] (1) {$1$};
					\node[inner, right of=1, label=45:$-1$] (2) {$0$};
					\node[inner, draw=none, above right=0.2 of 2] (2l) {};
					\draw (1) to [out=315, in=225, looseness=1] (2);
					\draw (2) to (2l);
				\end{tikzpicture}
			};
			\node[outer, below=.9 of A2] (B2) {
				\begin{tikzpicture}
					\node[inner, label=135:$0$] (1) {$1$};
					\node[inner, right of=1, label=45:$-1$] (2) {$0$};
					\node[inner, draw=none, above right=0.2 of 2] (2l) {};
					\draw (1) to [out=45, in=135, looseness=1] (2);
					\draw (2) to (2l);
				\end{tikzpicture}
			};
			\draw[->] (B1) -- (A1);		
			\draw[->] (B1) -- (A2);			\draw[->] (B2) -- (A1);			\draw[->] (B2) -- (A2);
	\end{tikzpicture}}
\end{center}
The action of $\Aut(\Gamma)=\SSS_2$ swaps the two edges. In particular, both the top dimensional cells $J^{1,-1}_{X,E(\Gamma)}$ and $J^{0,0}_{X,E(\Gamma)}$ are $\Aut(\Gamma)$-invariant, and contribute in the same way to the Hodge--Euler characteristic of the stratum. 
Moreover, since the Jacobian over $\M_{0,3}$ is trivial, we can reason as in Stratum~$\M_{1,2}/\SSS_2$ and obtain
$$
\e_{HE}^{\SSS_2}(J^{1,-1}_{\Gamma,E(\Gamma)})^{\SSS_2}=\e_{HE}^{\SSS_2}(J^{0,0}_{\Gamma,E(\Gamma)})^{\SSS_2}=L^4+L^3-L^2+1.
$$

The group $\SSS_2$ swaps the two edges $e_1$ and $e_2$ of $\Gamma$, and, hence, the two codimension $1$ elements, thus we can obtain their contribution by considering only one of them with trivial stabiliser. Again by the same computation as above, since the Jacobian over $\M_{0,3}$ is trivial, we have
$$
\e_{HE}^{\SSS_2}(J^{0,-1}_{\Gamma,\{e_1\}}\sqcup J^{0,-1}_{\Gamma,\{e_2\}})^{\SSS_2}=\e_{HE}(J^{0,-1}_{\Gamma,\{e_1\}})=L^3+2L^2-1.
$$

Summing twice the codimension $0$ contribution with the codimension $1$, we obtain
$$\e_{HE}(\ov J_\Gamma)=2L^4+3L^3+1.
$$
\item \textbf{Codimension $3$ strata.} As in the proof of Proposition~\ref{P:eJ2}, each $\M_\Gamma$ contributes for a copy of $L^3$ for each irreducible component of $\ov J_\Gamma$ up to the action of $\Aut(\Gamma)$. The graphs indexing the codimension $3$ strata are the following:
	\begin{center}\scalebox{.8}{
			\begin{tikzpicture}[remember picture, inner/.style={draw,circle,thick,node distance={15mm}},
				outer/.style={node distance={40mm}}]
				\node[outer] (A3) {
					\begin{tikzpicture}
						\node[inner] (1) {$0$};
						\node[inner, right of=1] (2) {$0$};
						\node[inner, draw=none, above right=0.2 of 1] (1l) {};
						\draw (1) to  (2);
						\draw (2) to [out=45, in=315, looseness=5] (2);
						\draw (1) to [out=135, in=225,looseness=5] (1);
						\draw (1) to (1l);
					\end{tikzpicture}
				};
				\node[outer, left of =A3] (A2) {
					\begin{tikzpicture}
						\node[inner] (1) {$0$};
						\node [inner, left of=1] (2) {$0$};
						\node [inner, right of=1] (3) {$1$};
						\node [inner, draw=none, above left=0.2 of 2] (2l) {};
						\draw (1) to [out=135, in=45, looseness=1] (2);
						\draw (1) to [out=225, in=315, looseness=1] (2);
						\draw (1) to  (3);
						\draw (2) to (2l);
					\end{tikzpicture}
				};		
				\node[outer, left of =A2] (A1) {
					\begin{tikzpicture}
						\node[inner] (1) {$0$};
						\node [inner, left of=1] (2) {$0$};
						\node [inner, right of=1] (3) {$1$};
						\node [inner, draw=none, above=0.2 of 1] (1l) {};
						\draw (2) to [out=135, in=225, looseness=5] (2);
						\draw (1) to  (2);
						\draw (1) to  (3);
						\draw (1) to (1l);
					\end{tikzpicture}
				};
				\node[outer, right of =A3] (A4) {
					\begin{tikzpicture}
						\node[inner] (1) {$0$};
						\node [inner, left of=1] (2) {$0$};
						\node [inner, draw=none, above left=0.2 of 2] (2l) {};
						\draw (1) to [out=135, in=45, looseness=1] (2);
						\draw (1) to [out=225, in=315, looseness=1] (2);
						\draw (1) to [out=45, in=315, looseness=5] (1);
						\draw (2) to (2l);
					\end{tikzpicture}
				};
				\node[outer, right of =A4] (A5) {
					\begin{tikzpicture}
						\node[inner] (1) {$0$};
						\node [inner, left of=1] (2) {$0$};
						\node [inner, draw=none, above left=0.2 of 2] (2l) {};
						\draw (1) to [out=135, in=45, looseness=1] (2);
						\draw (1) to [out=225, in=315, looseness=1] (2);
						\draw (1) to (2);
						\draw (2) to (2l);
					\end{tikzpicture}
				};
		\end{tikzpicture}}
	\end{center}
	We observe that the irreducible components of all the associated spaces $\ov J_\Gamma$ are $\Aut(\Gamma)$-invariant, hence their total contribution is $9L^3$, where the coefficient is the sum of the complexities of the graphs.
\end{itemize}
Adding up all the invariant Hodge--Euler characteristics we obtain 
$$\e_{HE}(\ov J_{2,1})=L^6+5L^5+15L^4+20L^3+\textup{ lower degree terms,}$$
hence the claim follows by Poincar\'e duality. This is independently verified with the complete computations over the remaining strata in the ancillary file.
\end{proof}

\begin{remark}
	Note that the results for the open strata $\M_2$ and $\M_{2,1}$ can be deduced from \cite[Proposition~E]{KannanSongVirtualHodgeNumbers}.\\
	The contributions of strata $\M_2$, $\M_{1,2}/\SSS_2$, $\M_{2,1}$ and $\M_{1,2}/\SSS_2$ can be alternatively computed using the relation between the cohomology of local systems on $\M_g$ and $\A_g$, studied by Getzler in \cite{GetzlerTopologicalRecursionGenus2,getzler1999eulercharacteristicslocalsystems} and Petersen in \cite{Petersen_2015}.
\end{remark}

By \cite[Lemma~3.11]{CanningLarson}, the following Lemma concludes the proof of Theorem~\ref{T:CohJ2}.
\begin{lemma}\label{L:chow}
	The fine compactified universal Jacobians $\ov J_2$ and $\ov J_{2,1}$ have the Chow--K\"unneth generation property (CKgP) (\cite[Definition~3.1]{CanningLarson}).
\end{lemma}
\begin{proof}
	All the unjustified statements in the proof follow from the Lemmas in \cite[Section~3]{CanningLarson}.
	
	Since the CKgP is compatible with stratifications and finite quotients, we only need to show that all the $\wt J_\Gamma$	 appearing in the stratification of $\ov J_2$ (resp. $\ov J_{2,1}$) have it. For any such $\Gamma$, $\wt J_{\Gamma}$ is stratified by extensions
	$$
	1\longrightarrow(\CC^*)^{b_1(G)}\longrightarrow J^{\un d}_{\Gamma,E(G)}\longrightarrow\prod_{v\in \Gamma} J^{ d(v)}_{g(v),n(v)}\longrightarrow 1,
	$$ 
	where $G\leq\Gamma$ is a connected spanning subgraph. Thus, the claim follows from showing that every $\wt J^{\un d}_{\Gamma,E(G)}$ has the CKgP. We observe that $\wt J^{\un d}_{\Gamma,E(G)}$ is an open subset of a projective bundle over $\prod_{v\in \Gamma} J^{ d(v)}_{g(v),n(v)}$, and hence has the CKgP if every $J^{ d'}_{g',n'}\cong J_{g',n'}$ appearing has it.  
	
	We observe that, for these universal Jacobians to appear, $(g',n')$ must satisfy $2g'+n'\leq 5$, and show that in this range any $J_{g',n'}$ has the CKgP.
	\begin{itemize}
		\item \textbf{$g=2$.} The Abel--Jacobi map $\C^2_2/\SSS_2\to J_2$ (resp. $\C_{2,1}^3/{\SSS_2}\to J_{2,1}$) in Diagram~\eqref{E:AJmapM2} (resp. Diagram~\eqref{E:AJmapM21}) is proper and surjective. Such universal curves are stratified by $\M_{2,m}$, with $m\leq3$, which have the CKgP by \cite[Theorem~1.1]{canning2024hyperell} and \cite[Section~6]{CanningLarson}. Hence $J_2$ and $J_{2,1}$ have it too.
		\item \textbf{$g=1$.} A universal Jacobian $J_{1,n'}$ is isomorphic to the universal curve $\C_{1,n'}=\M_{1,n'+1}\bigsqcup_{i=1}^{n'}\M_{1,n'}$, where $n'\leq 3$, and hence has the CKgP by \cite[Theorem~2.0.1]{Belorousski1998} and \cite[Section~5.2]{CanningLarson}.
		\item \textbf{$g=0$.} A universal Jacobian $J_{0,n'}$ is isomorphic to $\M_{0,n'}$ and hence has the CKgP as it is the complement of a hyperplane arrangement in $\AA^{n'-3}$.
	\end{itemize}
\end{proof}

\bibliographystyle{alpha}
\bibliography{bibtex}
\end{document}